\documentclass[a4paper,12pt,twoside,openright]{amsart}
\usepackage[english]{babel}
\usepackage[utf8]{inputenc}
\usepackage{fancyhdr}

\usepackage{amssymb}
\usepackage{amsmath} 
\usepackage{amsthm} 
\usepackage{enumitem}
\usepackage{array}
\usepackage[table]{xcolor}
\arrayrulecolor{blue}
\usepackage{tikz-cd}
\usepackage[all]{xy}
\usetikzlibrary{matrix, arrows, decorations.pathmorphing}
\usetikzlibrary{fit}
\usepackage{extarrows}
\usepackage{mathtools}
\usepackage{faktor}
\usepackage{float}

\makeatletter
\def\subsubsection{\@startsection{subsubsection}{3}%
  \z@{.5\linespacing\@plus.7\linespacing}{-.5em}%
  {\normalfont\bfseries}}
\makeatother

\theoremstyle{plain} 
\newtheorem{thm}{Theorem}[section]

\newtheorem*{mthm*}{Main Theorem}

\newtheorem{lem}[thm]{Lemma} 
\newtheorem{prop}[thm]{Proposition}
\newtheorem{propA}{Proposition}[]

\theoremstyle{definition} 
\newtheorem{defn}{Definition}[section]

\theoremstyle{remark} 
\newtheorem{oss}[thm]{Remark}

\makeatletter 
\newcommand{\xhooklongrightarrow}[2][]{%
  \ext@arrow3399{\hooklongrightarrowfill@}{#1}{#2}} 
\makeatother

 \DeclareMathOperator{\Ass}{Ass}

  \DeclareMathOperator{\car}{char}

   \DeclareMathOperator{\lt}{in_\prec}

 \DeclareMathOperator{\Min}{Min}

 \usepackage{blindtext}
\title{Binomial edge ideals of weakly closed graphs\\}
\author{Lisa Seccia}
\email{seccia@dima.unige.it}
\address{Universit\`a  di Genova,  Dipartimento di Matematica. 
 Via Dodecaneso 35, 16146 Genova, Italy}
 \thanks{The author is supported by PRIN  2020355BBY ``Squarefree Gr\"obner degenerations, special varieties and related topics".}
 \date{}
 
\begin{document}
 \maketitle
 
 \begin{abstract}
 Closed graphs have been characterized by Herzog et al. as the graphs whose binomial edge ideals have a quadratic Gr\"obner basis with respect to a diagonal term order.  In this paper, we focus on a generalization of closed graphs, namely weakly-closed graphs (or co-comparability graphs). 
Build on some results about Knutson ideals of generic matrices, we characterize weakly closed graphs as the only graphs whose binomial edge ideals are Knutson ideals for a certain polynomial $f$. In doing so, we re-prove Matsuda's theorem about the F-purity of binomial edge ideals of weakly closed graphs in positive characteristic and we extend it to generalized binomial edge ideals. Furthermore, we give a characterization of weakly closed graphs in terms of the minimal primes of their binomial edge ideals and we characterize all minimal primes of Knutson ideals for this choice of $f$.

 \end{abstract}
 
\section{Introduction}
Binomial edge ideals were introduced by Herzog, Hibi, Hreinsdóttir, Kahle and Rauh \cite{HHHKR} and, independently, by Ohtani \cite{Oh} as a generalization of determinantal ideals.  The main purpose was to understand the interplay between the combinatorial invariants of a graph and the algebraic invariants of its associated binomial edge ideal. It is worth mentioning that before binomial edge ideals were defined, other ways to encode combinatorial objects into ideals had been introduced. In this context, one of the greatest contributions came from Stanley and Reisner who established a bijection between simplicial complexes and squarefree monomial ideals. Also, Villarreal \cite{Vi} introduced the notion of edge ideals of graphs, which are ideals generated by monomials $x_ix_j$ corresponding to the edges of $G$.\\

 Let $R=\mathbb{K}[x_1, \ldots,x_n, y_1, \ldots,y_n]$ and let $G$ be a graph on $n$ vertices with edges $E(G)$. We write $J_G$ to denote the \emph{binomial edge ideal} of $G$, that is
\begin{equation*}
J_G= \left( f_{ij}:=x_i y_j-x_j y_i \mid \{i,j\} \in E(G)\right).
\end{equation*}

In other words, $J_G$ is the ideal generated by the $2$-minors of the generic matrix
$$ X_{n}=
\begin{bmatrix}
    x_1       & x_2 & x_3 & \dots & x_n \\
     y_1       & y_2 & y_3 & \dots & y_n
\end{bmatrix}
$$

whose column indices are given by the edges of $G$. \par 
If we take $G=K_n$ to be the complete graph on $n$ vertices, then it is clear from the definition that its binomial edge ideal $J_{K_n}$ is the ideal of $2$-minors of $X_n$. This is why binomial edge ideals can be considered a generalization of determinantal ideals.\par

In \cite{HHHKR} the authors study the algebraic properties of binomial edge ideals in terms of combinatorial invariants of the underlying graph. Among other things, they  prove that binomial edge ideals are radical, and they give a combinatorial description of their minimal primes (see \cite[Theorem 3.2]{HHHKR}). Furthermore, they find out that the only graphs whose binomial edge ideals have a quadratic Gr\"obner basis (with respect to a diagonal term order) are those such that if $\{i,k\} \in E(G)$ then $\{i,j\} \in E(G)$ and $\{j,k\} \in E(G)$ for all integers $1 \leq i <j<k\leq n$. They called them \emph{closed graphs}.
\begin{thm}\cite[Theorem 1.1]{HHHKR}\label{thmHHH}
$G$ is a closed graph if and only if the natural generators of $J_G$ form a Gr\"obner basis with respect to a diagonal term order.
\end{thm}

A generalization of this theorem to determinantal facet ideals of simplicial complexes has recentely been found in \cite[Theorem 82 and 87]{BSV}. These results correct a theorem stated in \cite{EHHM} and provide a partial answer to a question by Almousa–Vandebogert \cite{AV}.\\

Later, Matsuda \cite{Ma} extends the algebraic approach  introduced by Herzog et al. in \cite{HHHKR} to a larger class of graphs, that he called \emph{weakly closed graphs}.
\begin{defn} Let $G$ be a simple graph on $[n]$. $G$ is said to be \emph{weakly closed} if there exists a labeling of the vertices such that for all integers $1 \leq i <j<k\leq n$, if $\{i,k\} \in E(G)$ then $\{i,j\} \in E(G)$ or $\{j,k\} \in E(G)$.
\end{defn}

Weakly closed graphs are a generalization of closed graphs. In fact, while the definition of closed graphs requires that both $\{i,j\}$ and  $\{j,k\}$ are edges of $G$, for weakly closed graphs it is enough that one of them is an edge of $G$.

\begin{oss}
It is worth pointing out that closed graphs and weakly-closed graphs were already well-known and widely studied in combinatorics where it is costumary to refer to them as \emph{unit-interval graphs} and \emph{co-comparability graphs} (i.e. graphs whose complement is the comparability graph of a poset \cite[Theorem 1.9]{Ma}). So, it would be more accurate to say that these graphs were re-discovered by Herzog et al. and Matsuda from an algebraic perspective. 
\end{oss}

With the above in mind, it is easy to see that complete multipartite graphs and interval graphs are weakly closed and that weakly closed graphs are perfect (see \cite{Ma} for more details).\\

In \cite{Ma} Matsuda began the study of binomial edge ideals of weakly-closed graphs. In particular, assuming that $\mathbb{K}$ has positive characteristic, he generalized Othani's theorem about $F$-purity of binomial edge ideals associated with complete multipartite graphs (see \cite[Theorem 3.1]{Oh}).

\begin{thm}\cite[Theorem 2.3]{Ma}\label{Mafp}
Let $G$ be a weakly closed graph and let $J_G$ be the binomial edge ideal associated with $G$. Then $R/J_G$ is $F$-pure.
\end{thm}

The above result, together with Theorem \ref{thmHHH}, motivates us to continue the work of Matsuda on binomial edge ideals associated with weakly-closed graphs, with a special focus on their interaction with Knutson ideals (for a short overview of Knutson ideals, see Section 2). In particular, similarly to what has been done in \cite{HHHKR} for closed graphs, we give an algebraic characterization of weakly closed graphs in terms of their binomial edge ideals.

\begin{mthm*}[Theorem \ref{cf-wc}] $G$ is a weakly closed graph on $[n]$ if and only if (there exists a labeling such that) $J_G$ is a Knutson ideal associated with $f=y_1 f_{12} \ldots f_{n-1n}y_n  \in R$.
\end{mthm*}

To prove this theorem, we start off by proving that binomial edge ideals of weakly closed graphs are Knutson ideals. By the properties of Knutson ideals, the following result also gives an alternative proof of Theorem \ref{Mafp} in positive characteristic:

\begin{propA}[Proposition \ref{wc-cf}]\label{introwc-cf}
Let $G$ be a weakly closed graph on $[n]$. Then its binomial edge ideal $J_G$  is a Knutson ideal. \par 
In particular, if $\mathbb{K}$ has positive characteristic, then $\mathbb{K}[X]/J_G$ is $F$-pure.
\end{propA}

Moreover, the proof of this result suggests the following characterization of weakly closed graphs in terms of minimal primes of their binomial edge ideals.

\begin{propA}[Proposition \ref{psps}] \label{introprimes}$G$ is a weakly closed graph if and only if the minimal primes of its binomial edge ideal can be written as a sum of determinantal ideals on (disjoint) adjacent columns.
\end{propA}

The advantage of this alternative proof of Matsuda's theorem is that it easily extends to a larger class of ideals, introduced by Rauh \cite{Ra} and called \emph{generalized binomial edge ideals}. Thus we get the following:

\begin{propA}[Proposition \ref{wc-cf-gen}]
Let $G$ be a weakly closed graph on $[n]$. Then its generalized binomial edge ideal $\mathfrak{J}_G$  is a Knutson ideal. \par 
In particular, if $\mathbb{K}$ has positive characteristic, then $\mathbb{K}[X]/\mathfrak{J}_G$ is $F$-pure.
\end{propA}


In Section 4 we use Proposition \ref{introprimes} to prove that the converse of Proposition \ref{introwc-cf} is still true. This completes the proof of the main theorem, i.e. the set of binomial edge ideals in $\mathcal{C}_f$ coincides with that of binomial edge ideals of weakly closed graphs (Theorem \ref{cf-wc}).\par 
For this purpose, we will also need the following description of all minimal primes of Knutson ideals associated with $f=y_1 f_{12} f_{23} \ldots f_{n-1n}y_n  \in R$.

\begin{propA}[Proposition \ref{PrIdCf}]\label{introMP} Let $I$ be a Knutson ideal associated with $f$ and let $P \in \Min (I)$. Then
$$P= \left(\left(y_1,\ldots,y_{k-1}\right)+\left(x_u\right)_{U \subset \{1,\ldots, k-1\}}\right)+L+\left(\left(x_{l+1},\ldots,x_n\right)+\left(y_v\right)_{V \subset \{l+1,\ldots, n\}}\right)$$
where $L\subset \mathbb{K}[x_k,x_{k+1},\ldots, x_l]$ is a minimal prime of the binomial edge ideal of a weakly closed graph and each of the three summands may possibly be the zero ideal.
\end{propA}

The proof of Proposition \ref{introMP} is quite technical. However, it is worth noticing that knowing all minimal primes of Knutson ideals associated with $f$ is a quite strong result in the study of Knutson ideals of generic matrices. In fact, Knutson ideals are radical and minimal prime ideals can be considered as the \lq \lq building blocks" of radical ideals. Hence, heuristically, having a characterization of all minimal primes of Knutson ideals is not so far from having a characterization of Knutson ideals themselves.  In this sense, this is a step forward towards a partial solution of the problem of finding a characterization of all Knutson ideals of generic matrices (see \cite{Se2}), in the specific case of $2 \times n$ generic matrices. The general result could have interesting consequences on Gr\"obner bases of determinantal-like ideals.\\

\textbf{Acknowledgements.} I would like to thank Matteo Varbaro for several enlightening and helpful discussions.  I am also grateful to Bruno Benedetti for some observations on the combinatorial aspects of the topic, and on the connection with determinatal facet ideals.


\section{An overview of Knutson ideals}

Knutson ideals were first introduced by Conca and Varbaro in \cite{CV} and they were named after Knutson's work \cite{Kn} on compatibly split ideals and degeneration.
 
 \begin{defn}[Knutson ideals] \label{K.I.} Let $f \in S= \mathbb{K}[x_1,\ldots,x_n]$ be a polynomial such that its leading term $\lt (f)$ is a squarefree monomial  for some term order $\prec$ .
Define $\mathcal{C}_f$ to be the smallest set of ideals satisfying the following conditions:
\begin{enumerate}
\item[1.] $(f) \in \mathcal{C}_f$;
\item[2.]  If $I \in \mathcal{C}_f$ then $I:J \in \mathcal{C}_f$ for every ideal $J \subseteq S$;
\item[3.] If $I$ and $J$ are in $\mathcal{C}_f$ then also $I+J$ and $I \cap J$ must be in $\mathcal{C}_f$.
\end{enumerate} 
If $I$ is an ideal in $\mathcal{C}_f$, we say that I is a \emph{Knutson ideal associated with} $f$. More generally, we say that $I$ is a \emph{Knutson ideal} if $I \in \mathcal{C}_f $ for some $f$.
 \end{defn}
 
 By the assumption on its initial form, if $\deg f=n$, then $f$ defines a splitting map on $S$. Therefore, Knutson ideals turn out to be compatibly split ideals with respect to this map, in the sense of Schwede \cite{Sc}. As a consequence, they define $F$-pure rings in positive characteristic. In addition, Knutson ideals
have square-free initial ideals (see \cite{Kn},\cite{Se1}). Hence, the extremal Betti numbers, Castelnuovo-Mumford regularity, and depth of Knutson ideals coincide
with the corresponding numerical invariants of their initial ideals (see \cite{CV}). Lastly, Gr\"obner bases of Knutson ideals behave \lq\lq well\rq \rq with respect to sums. This fact makes computations of Gr\"obner bases easier in many cases. Hence, Knutson ideals also provide a useful tool to solve problems in applied algebra. \par 
For all these reasons, these ideals are objects of particular interest in computational algebra, combinatorial commutative algebra, and tight closure theory. More details about the properties of Knutson ideals in any characteristic can be found in \cite{Se1}. These properties were first proved by Knutson \cite{Kn} in the case $\mathbb{K}=\mathbb{Z}/p\mathbb{Z}$. 

\begin{oss}
 It is useful to point out that since every ideal of $\mathcal{C}_f$ is radical, the second condition in Definition \ref{K.I.} can be replaced by the following:
 \begin{itemize}
 \item[$2^\prime .$] If $I \in \mathcal{C}_f$ then $\mathcal{P} \in \mathcal{C}_f$ for every $\mathcal{P} \in \Min(I)$.\\
 \end{itemize}
 \end{oss}
 
 It has already been shown that some interesting classes of ideals are Knutson ideals; for instance  determinantal ideals of Hankel matrices \cite{Se1}  and generic matrices \cite{Se2}. In this paper we prove that also binomial edge ideals of weakly closed graphs are Knutson ideals. Actually, we prove a stronger result: a graph is weakly closed if and only if its binomial edge ideal is a Knutson ideal.\\

In order to prove the main theorem of this paper, we first need to specify the polynomial $f$ defining the family of Knutson ideals we are considering. Since binomial edge ideals are a generalization of the ideal of 2-minors of a generic matrix of size $2 \times n$, the choice of $f$ is going to be the standard one for generic matrices of any size (see \cite[Theorem 2.1]{Se2}), as explained  below. \\

Given a graph $G$ on $n$ vertices and $X_{n}$ the generic matrix of size $2 \times n$ 

\[ X_{n}=
\begin{bmatrix}
    x_1       & x_2 & x_3 & \dots & x_n \\
     y_1       & y_2 & y_3 & \dots & y_n
\end{bmatrix}
\]

we define 
$$f=y_1  f_{12}  f_{23} \cdots f_{n-1 n}  x_n \in R.$$

In other words $f$ is the product of all the minors of $X$ whose diagonals are the diagonals of $X$. Thus, if we equip $R$ with a diagonal term order, we get

$$\lt (f)= \prod \limits_{i=1}^{n} x_i y_i.$$

We can then construct the Knutson family of ideals associated with this $f$. This choice of $f$ allows us to apply some known results from \cite{Se2}. In particular, in the next section we make use of the following lemma.


\begin{lem} \cite[Theorem 2.1]{Se2}\label{lemSe2}
 Let $X=X_n$ and $f$ be as in the previous discussion. Denote by $X_{[a,b]}$ the submatrix of $X$ with adjacent columns from $a$ to $b$. 
 Then $$I_{t}(X_{[a,b]}) \in \mathcal{C}_f$$ for every $1 \leq a \leq b \leq n$ and  $t \in \{1,2\}$ .
\end{lem}

More generally, this result can be stated for any generic matrix of size $m \times n$ and for any size of the minors: given a generic matrix, all determinantal ideals on adjacent columns (or adjacent rows) are Knutson ideals for the standard choice of $f$ (see \cite[Theorem 2.1]{Se2} for more details).

\section{(Generalized) binomial edge ideals of weakly closed graphs}
 
We start off by proving that the binomial edge ideal attached to a weakly closed graph is a Knutson ideal.

\begin{prop}\label{wc-cf} Let $G$ be a weakly closed graph on the vertex set $\{1,\ldots,n\}$ and let $f= y_1  f_{12} f_{23}\cdots f_{n-1 n}  x_n \in R$. Then (there exists a labeling of the vertices such that) $$J_G \in \mathcal{C}_f.$$
In particular, if $\mathbb{K}$ has positive characteristic then $R/J_G$ is $F$-pure.
\end{prop}

The strategy to prove Proposition \ref{wc-cf} is to write each of the minimal primes of the binomial edge ideal $J_G$ as a sum of determinantal ideals on adjacent columns, so that we can apply Lemma \ref{lemSe2}.

\begin{proof}[Proof of Proposition \ref{wc-cf}]
For each subset $S\subset [n]$ and $T= [n] \setminus S$, define $G_1, \ldots, G_{c(S)}$ to be the connected components of $G_T$ (i.e. the restriction of $G$ to $T$) and let $\widetilde{G}_1, \ldots,\widetilde{G}_{c(S)}$ be the corresponding complete graphs on their vertices. Set 
$$P_S:= \left( \bigcup_{i \in S} \{x_i,y_i\}\right)+ J_{\widetilde{G}_1}+\ldots+ J_{\widetilde{G}_{c(S)}} .$$ \par
$P_S$ is a prime ideal and it has been shown (see \cite{HHHKR}) that the primary decomposition of the binomial edge ideal associated with $G$ is given by
$$J_G= \bigcap_{S \subseteq [n]} P_S.$$\par
If we prove that $P_S$ is a Knutson ideal for each $S$, we get that $J_G \in \mathcal{C}_f$.\par
First of all, notice that by Lemma \ref{lemSe2}, we know that $\left( x_s,y_s \right) \in \mathcal{C}_f$ for every $s \in S$, because it is the ideal  of $1$-minors on column $s$. So their sum $\left( \bigcup_{i \in S} \{x_i,y_i\}\right)$ is also in $\mathcal{C}_f$. Unfortunately, the lemma does not apply to $J_{\widetilde{G}_{i}}$ because the vertices of $G_i$ might not be consecutive. Thus, we need to reduce to the case where each of the $J_{\widetilde{G}_{i}}$'s is an ideal of minors on adjacent columns of $X$ (equivalently, it is the binomial edge ideal of a complete graph on consecutive vertices) so that we can apply Lemma \ref{lemSe2}.\par 
 For this purpose, fix $i \in \{1, \ldots,c(S)\}$ and let $V(G_i)=V(\widetilde{G}_{i}):= \left\lbrace j_1, \ldots, j_{t_i}\right\rbrace$.\par
If $j_{k+1}=j_{k}+1$ for every $k=1, \ldots, t_i-1$, that is, if the vertices of $G_i$ are consecutive, then $J_{\widetilde{G}_{i}}=I_2 (X_{[j_1,j_{t_i}]}) \in \mathcal{C}_f$ by Lemma \ref{lemSe2}. \par 
Assume instead that $j_k-j_{k-1} >1$ for some $k$ and let $l$ be a vertex of $G$ such that $j_{k-1}<l<j_k$. Since $G_i$ is connected, there exist $m,n \in V(G_i)$ such that $m<l<n$ and $\{m,n\} \in E(G_i)\subset E(G)$. This implies that either $\{m,l\} \in E(G)$ or $\{l,n\} \in E(G)$, because $G$ is weakly closed. Assume that $l \notin S$. Then $l$ must be in the same connected component of $m$ and $n$. This would imply that $l \in  V(G_i)$, a contradiction. In other words, we have just shown that if there is a \lq\lq gap" between the vertices of the connected component $G_i$, every vertex $l$ in this gap must be in $S$. But then, we can add all these missing vertices to $V(G_i)$ and replace  $J_{\widetilde{G}_i}$ with $\overline{J}_{\widetilde{G}_i}:=I_2(X_{[j_1, j_{t_i}]})$ without changing our prime ideal $P_S$. Indeed, in doing so, we are adding some $2$-minors to our original prime ideal, but these minors were already contained in $P_S$, since they are contained in the ideal $\left( \bigcup_{i \in S} \{x_i,y_i\}\right)$.\par 
In conclusion, if we replace each of the $J_{\widetilde{G}_i}$'s with its corresponding $\overline{J}_{\widetilde{G}_i}:=I_2 ( X_{[a_i, b_i]})$, we get 
$$P_S= \left( \bigcup_{s \in S} \{x_s,y_s\}, J_{\widetilde{G}_1}, \ldots, J_{\widetilde{G}_{c(S)}} \right)= \left( \bigcup_{s \in S} \{x_s,y_s\}, I_2 ( X_{[a_1, b_1]}), \ldots, I_2 ( X_{[a_{c(S)}, b_{c(S)}]}) \right).$$
Finally, we can apply Lemma \ref{lemSe2} and we obtain that $P_S \in \mathcal{C}_f$ for every $S \subset [n]$, because it is the sum of ideals of minors on adjacent columns. This completes the proof.
\end{proof}

Using the same notation as in the proof of Proposition \ref{wc-cf}, we set 
$$\overline{P}_S:= \left( \bigcup_{s \in S} \{x_s,y_s\}, \overline{J}_{\widetilde{G}_1}, \ldots, \overline{J}_{\widetilde{G}_{c(S)}} \right)$$ 
where each of the $\overline{J}_{\widetilde{G}_i}$'s is an ideal of 2-minors on adjacent columns. It has just been shown that if $G$ is a weakly-closed graph, each of the minimal primes of its binomial edge ideal is a sum of determinantal ideals on adjacent columns, that is $P_S=\overline{P}_S$ for every minimal prime of $J_G$. Actually, it turns out that the only binomial edge ideals whose minimal primes can be written in this way are those arising from weakly closed graphs.\\

 We have the following characterization of weakly closed graphs in terms of minimal primes of their binomial edge ideals.

\begin{prop}\label{psps}Let $G$ be a connected graph and let $J_G= \bigcap_{S\subseteq [n]} P_S$ be the primary decomposition of the binomial edge ideal associated with $G$. The following are equivalent
\begin{itemize}
\item[(1)]$G$ is weakly closed. 
\item[(2)]There exists a labeling of the vertices of $G$ such that $P_S=\overline{P}_S$ for every minimal prime of $J_G$.
\end{itemize}
\end{prop}

\begin{proof}
$(1) \Rightarrow (2)$ has already been proved. \par 
$(2) \Rightarrow (1)$ Let us assume by contradiction that $G$ is not weakly closed. We will show that for every labeling of the vertices it is possible to find $S\subseteq [n]$ such that $P_S \neq \overline{P}_S$ and we are done.\par
Since $G$ is not weakly closed, for each labeling of the vertices there exist $k,l,m \in V(G)$ with $k<l<m$ such that $\{k,m\}\in E(G)$ and $\{k,l\},\{l,m\} \notin E(G)$. Nonetheless, there are a finite number of paths of length $\geq 2$ connecting $k$ to $l$ in $G \setminus \{m\}$ and $m$ to $l$ in $G \setminus \{l\}$. We will denote them as follows:

\begin{align*}
p_1: \quad k,p_{11}&,p_{12},\ldots,l\\
p_2: \quad k,p_{21}&,p_{22},\ldots,l\\
 & \vdots\\
p_r: \quad k,p_{r1}&,p_{r2},\ldots,l\\
q_1:\quad m,q_{11}&,q_{12},\ldots,l\\
q_2: \quad m,q_{21}&,q_{22},\ldots,l\\
&\vdots \\
q_t: \quad m,q_{t1}&,q_{t2},\ldots,l.
\end{align*}  
Now, take $S$ to be the set of the first vertices of the previous paths, that is
 $$S=\{ p_{11},p_{21}, \ldots, p_{r1},q_{11},q_{21}, \ldots,q_{t1}\}.$$ Then $P_S$ is a minimal prime of $J_G$ and we claim that $P_S \neq \overline{P}_S$.\par 
 By definition of $S$, $\{k,m\}$ and $l$ do not belong to the same connected component of  $G_{[n] \setminus S}$. Without loss of generality we can assume that $G_1$ is the connected component of $\{k,m\}$ and $G_2$ is the connected component of $l$. Observe that $x_l,y_l \notin \bigcup_{s \in S} \{x_s,y_s\}$. Thus, it is straightforward to see that 
 $$\left( \bigcup_{s \in S} \{x_s,y_s\}, J_{\widetilde{G}_1}, J_{\widetilde{G}_2} \right) \neq \left( \bigcup_{s \in S} \{x_s,y_s\}, \overline{J}_{\widetilde{G}_1}, \overline{J}_{\widetilde{G}_{2}} \right).$$
 This shows that $P_S \neq \overline{P}_S$.
\end{proof}

The proof of previous results relies on two main ingredients: the primary decomposition of binomial edge ideals, and Lemma \ref{lemSe2} on determinantal ideals on adjacent columns. Actually, these two facts hold in a more general setting. Thus Propositions \ref{wc-cf} and \ref{psps} extend to  a larger class of ideals, called \emph{generalized binomial edge ideals}. These ideals were introuced by Rauh \cite{Ra} as a generalization of binomial edge ideals. \par 
Let $G$ a graph on $n$ vertices and let $X$ be the generic matrix of size $m \times n$ with entries $x_{ij}$. We can attach to $G$ an ideal in the polynomial ring $\mathbb{K}[X]=\mathbb{K}[x_{ij} \mid  1 \leq i \leq m, 1 \leq j \leq n]$ defined as follows, 
\begin{equation*}
\mathfrak{J}_G= \sum \limits_{\{i,j\} \in E(G)} I_2 \begin{pmatrix}
    x_{1i}      & x_{1j} \\
    x_{2i} & x_{2j}\\
    \vdots & \vdots\\
     x_{mi}       & x_{mj}
\end{pmatrix} .
\end{equation*}
$\mathfrak{J}_G$ is the \emph{generalized binomial edge ideal of }$G$. If $m=2$ this definition boils down to the usual definition of binomial edge ideal.\par
In \cite{Ra} the author proves that generalized binomial edge ideals are radical and the primary decomposition is analogous to that of classical binomial edge ideals. Furthermore, as already said, Lemma \ref{lemSe2} holds for any generic matrix of size $m \times n$ and for any size of the minors (see \cite[Theorem 2.1]{Se2}). Hence, if $G$ is a weakly closed graph, the proofs of Proposition \ref{wc-cf} and \ref{psps} easily generalize to $\mathfrak{J}_G$.

\begin{prop}\label{wc-cf-gen} Let $G$ be a weakly closed graph on $[n]$. Define $f \in \mathbb{K}[X]$ to be the product of all the minors of $X$ corresponding to all the diagonals of $X$ (see \cite[Theorem 2.1]{Se2} for further details). Then $\mathfrak{J}_G \in \mathcal{C}_f$. \par 
In particular, if $\mathbb{K}$ has positive characteristic, then $\mathbb{K}[X]/\mathfrak{J}_G$ is $F$-pure.
\end{prop}

\begin{prop}\label{psps-gen}Let $G$ be a connected graph and let $\mathfrak{J}_G= \bigcap_{S\subseteq [n]} \mathfrak{P}_S$ be the primary decomposition of the  generalized binomial edge ideal associated with $G$. Then $G$ is weakly closed if and only if there exists a labeling of the vertices of $G$ such that $\mathfrak{P}_S$ is a sum of determinantal ideals on adjacent columns.
\end{prop}

\section{Characterization of Knutson binomial edge ideals}
The goal of this section is to use the characterization of weakly closed graphs found in Proposition \ref{psps} to prove the converse of Proposition \ref{wc-cf}. Thereby we find an algebraic characterization of weakly closed graphs.

\begin{thm}\label{cf-wc}
$G$ is weakly closed $\Leftrightarrow$ ($\exists$ a labeling of the vertices such that) $J_G \in \mathcal{C}_f$.
\end{thm}

So far, we have the following:
%
 $$\xymatrix@=3em{
G \text{ weakly closed}\ar@{<=>}[rr]\ar@{=>}[d]\ar@{=>}[rrd]^{ \ \ \text{  Matsuda }}_{\car \mathbb{K}=p>0}&& P_S=\overline{P}_S,  \ \forall P_S \in \Min (J_G)\\
J_G \in \mathcal{C}_f\ar@{=>}[rr]_{\car \mathbb{K}=p>0}&& J_G \text{ $F$-pure}.
}$$


To prove Theorem \ref{cf-wc} we first characterize all the minimal primes of the ideals in $\mathcal{C}_f$.
\begin{prop} \label{PrIdCf}
Let $I\in \mathcal{C}_f$ and let $P \in \Min (I)$. Then
$$P= \left(y_1,\ldots,y_{k-1}\right)+\left(x_u\right)_{U \subset \{1,\ldots, k-1\}}+L_S +\left(x_{l+1},\ldots,x_n\right)+\left(y_v\right)_{V \subset \{l+1,\ldots, n\}}$$
where $L_S=\overline{L}_S\subseteq \mathbb{K}[x_{k},y_{k},\ldots,x_{l},y_{l}]$ is a minimal prime of a weakly closed graph, and each of the three summands may possibly be the zero ideal.
\end{prop}

This characterization, together with \cite[Theorem 3.2]{HHHKR}, enables us to prove Theorem \ref{cf-wc}. In fact, it shows that among these minimal primes, those which are minimal primes of binomial edge ideals have the property described in Proposition \ref{psps}. Hence, the underlying graphs must be weakly-closed.\\

Since the proof of Proposition \ref{PrIdCf} is fairly technical, we defer it until after the proof of the main theorem of this section.

\begin{proof}[Proof of Theorem \ref{cf-wc}]
We have already proved in Proposition \ref{wc-cf} that if $G$ is weakly closed, $J_G$ is a Knutson ideal associated with $f$. It remains to prove that if $J_G$ is a Knutson binomial edge ideal, then $G$ is a weakly closed graph. \par 
Let $J_G$ be a binomial edge ideal in $\mathcal{C}_f$. By \cite[Theorem 3.2]{HHHKR}, we know that a primary decomposition of $J_G$ is given by 
$$J_G= \bigcap_{S \subseteq [n]} P_S$$
where $P_S:= \Bigl( \bigcup_{i \in S} \{x_i,y_i\}\Bigr)+ J_{\widetilde{G}_1}+\ldots+ J_{\widetilde{G}_{c(S)}}.$ On the other hand, by Proposition \ref{PrIdCf}, we know that $P_S =\overline{P}_S$ for every minimal prime of $J_G$. Hence, the thesis follows from Proposition \ref{psps}.
\end{proof}

In view of the proof of Proposition \ref{PrIdCf} we first recall a property of Knutson ideals that we are going to use throughout this section  in some more or less evident form .
 \begin{oss}\label{sumint}
 If $I,J,K$ are Knutson ideals, then sum distributes over intersection:
 $$I +(J \cap K)=(I+J) \cap (I+K).$$
 This fact easily follows from \cite[Remark 1]{Se1} and from the fact that the union of Gr\"obner bases of Knutson ideals associated with $f$ is a
Gr\"obner basis of their sum.
 \end{oss}

By definition, $\mathcal{C}_f$ is constructed from $(f)$ by taking its minimal primes, their sums, their intersections, and iterating. Since by Remark \ref{sumint} finite sums and intersections commute, we only need to prove Proposition \ref{PrIdCf} for sums of minimal prime ideals of Knutson ideals.\par 
Furthermore, by definition, the sum of two Knutson ideals is again a Knutson ideal. Since we know that binomial edge ideals of weakly closed graphs are in $\mathcal{C}_f$, so are their minimal primes and the sums of these minimal primes. However, the sum of two prime ideals could be not prime. Hence, it is clear that one of the first thing we need to check in order to prove Proposition \ref{PrIdCf} is the following:
\begin{lem}\label{lemmaP+Q}
Assume that $P$ and $Q$ are minimal primes of the binomial edge ideals of two weakly closed graphs, so that $P=\overline{P}$ and $Q=\overline{Q}$. Then every minimal prime $L$ of the sum $P+Q$ has the property $L=\overline{L}$ described in the previous section.
\end{lem}

Before proving the lemma, we make the following observation about the structure of binomial edge ideals with two associated primes.

\begin{oss}
Let $P$ be a minimal prime of a binomial edge ideal on $n$ vertices, then it has the form 
$$P=\left( \bigcup_{i \in S} \{x_i,y_i\}, J_{G_1}, \ldots, J_{G_t} \right)$$ where each $G_i$ is a complete graph with vertex set $V_i$. Denote by $\widetilde{V}$ the set of vertices that do not appear in $P$. Then $I_2(X_{[1,n]}) \cap P$ is the primary decomposition of the binomial edge ideal of the graph
$$G= K_{S} \cup K_{S,\widetilde{V}}\cup K_{S,V_1} \cup \ldots \cup K_{S,V_t}\cup G_1 \cup \ldots \cup G_t $$ 
 where $K_S$ denote the complete graph on $S$ and $K_{S,V_i}$ denote the complete bipartite graph on $S$ and $V_i$.\par 
 Actually, in \cite{Sh} Sharifan proved that if $G$ is a connected graph on $n$, then
 $\Ass|(J_G)|=2$ if and only if $G$ is the join of a complete graph $G_1$ and  a graph $G_2$ which is a disjoint union of complete graphs.\par 
 Let us take as an example  $P= (x_3,y_3, I_2(X_{[4,6]}))$. Then $S=\{3\}$, $V_1=\{4,5,6\}$ and $\tilde{V}= \{1,2\}$. Hence,
 $$G=K_{3,V_1} \cup K_{3,\widetilde{V}} \cup K_{[4,6]}$$
\begin{figure}[htbp]
\centering
\includegraphics[scale=0.66]{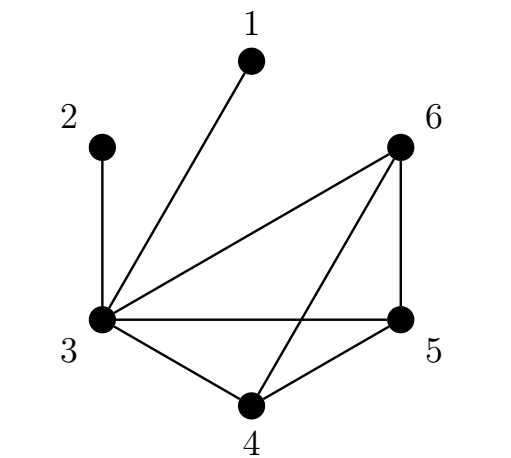} 
\end{figure}

and $P \cap I_2(X_{[1,6]})= J_{G}=([3,6],[3,5],[3,4],[2,3],[1,3],[5,6],[4,6],[4,5])$. 
 \end{oss}

With this in mind we can now prove Lemma \ref{lemmaP+Q}.

\begin{proof}
By assumption, there exist $G_1$ and $G_2$ weakly closed graphs such that $P \in \Min (J_{G_1})$ and $Q \in \Min (J_{G_2})$.
Let $P+Q=L_1\cap \ldots\cap \L_t$ be the minimal primary decomposition of $P+Q$. We want to prove that the $L_i$ are minimal prime ideals of the binomial edge ideal of a weakly closed graph, so that $L_i=\overline{L_i}$ for every $i$.\par
Note that we can always choose $n$ big enough such that $L_i\nsupseteq I_2(X_{[1,n]})$ for every $i\in \{1, \ldots,t\}$. Let $X=X_{[1,n]}$.\par   
Assume for the moment that $P,Q \subseteq I_2(X)$, then $P$ and $Q$ must not contain variables, that is
\begin{align*}
P&= J_{\overline{G}_P}\\
Q&=J_{\overline{G}_Q}
\end{align*}

where $\overline{G}_P$ and $\overline{G}_Q$ are unions of disjoint complete graphs on consecutive vertices (because $G_1$ and $G_2$ are weakly closed). Hence
$$P+Q= J_{\overline{G}_P}+J_{\overline{G}_Q}= J_{\overline{G}_P\cup \overline{G}_Q}.$$
Being $\overline{G}_P\cup \overline{G}_Q$ a weakly closed graph, $L=\overline{L}$ for every ideal $L \in \Min (P+Q)$.\par
Now assume without loss of generality that $P \nsubseteq I_2(X)$. 
 By Theorem \ref{wc-cf}, we know that $J_{G_1}, J_{G_2} \in \mathcal{C}_f$ and so are $P$ and $Q$. Hence $P+Q \in \mathcal{C}_f$. Now we consider the intersections 
\begin{align*}
I_2(X) &\cap P \\
I_2(X) &\cap Q.
\end{align*}
By the previous remark, we know that these are binomial edge ideals. Furthermore, being $P=\overline{P}$ and $Q=\overline{Q}$, these intersections are binomial edge ideals of two weakly closed graphs, say $\widetilde{G}_1$ and $\widetilde{G}_2$. Again by Theorem \ref{wc-cf}, $J_{\widetilde{G}_1}$ and $J_{\widetilde{G}_2}$ are Knutson ideals of $f$, so $$J_{\widetilde{G}_1 \cup \widetilde{G}_2}=J_{\widetilde{G}_1}+J_{\widetilde{G}_2} \in \mathcal{C}_f.$$
By Remark \ref{sumint}
 
\begin{equation*}
\begin{split}
J_{\tilde{G}_1 \cup \tilde{G}_2}=J_{\tilde{G}_1}+J_{\tilde{G}_2}&=( I_2(X) \cap P)+( I_2(X) \cap Q)\\
&=I_2(X) \cap (I_2(X) +Q) \cap (I_2(X) +P) \cap (P+Q)\\
&= I_2(X) \cap (P+Q)\\
&= I_2(X) \cap (L_1 \cap \ldots \cap L_t)
\end{split}
\end{equation*}

If this were the minimal primary decomposition of $J_{\tilde{G}_1 \cup \tilde{G}_2}$, then $L_1 , \ldots,L_t$ would be minimal prime ideals of a binomial edge ideal of a weakly closed graph, hence the thesis.\par 
Since we have assumed that $P \nsubseteq I_2(X)$, clearly $I_2(X) \nsupseteq L_1 \cap \ldots\cap L_t=P+Q$. Moreover $L_i \nsupseteq I_2(X) \cap L_1 \cap \ldots\cap L_{i-1} \cap L_{i+1} \cap \ldots\cap  L_t $, otherwise $L_i$ would contain either $I_2(X)$ or $L_j$ for some $j\neq i$, but this is impossible by the choice of $X$ and the fact that the $L_i$ are minimal primes of $P+Q$. This shows that $I_2(X) \cap (L_1 \cap \ldots \cap L_t)$ is the minimal primary decomposition of $J_{\tilde{G}_1 \cup \tilde{G}_2}$ and we are done.
\end{proof}

To construct $\mathcal{C}_f$, we start from the ideal $$(f)=(x_n f_{12}\cdots f_{n-1,n} y_1)$$ and we take its minimal primes. Thus we obtain the following ideals:
$$(x_n), (f_{12}),(f_{23}),\ldots,(f_{n-1n}),(y_1).$$
Among them the only binomial edge ideals are those of the form $(f_{i,i+1})$, which corresponds to the graph with exactly one edge, namely $\{i,i+1\}$, which is clearly weakly closed. If we take the sum of these binomial edge ideals, we obtain binomial edge ideals of (union of) paths on consecutive vertices. We can then consider their associated primes, the sum of these primes, and the intersections. Iterating this procedure, by Lemma \ref{lemmaP+Q} and Remark \ref{sumint}, the prime ideals we obtain are always of the form 
$$P_S=\overline{P}_S= \left( \bigcup_{s \in S} \{x_s,y_s\}, I_2 ( X_{[a_1, b_1]}), \ldots, I_2 ( X_{[a_{c(S)}, b_{c(S)}]}) \right).$$
In particular, if the intersection of these primes is a binomial edge ideal $J_G$, $G$ must be weakly closed by Proposition \ref{psps}.\\

It remains to investigate the case when we start from an ideal that contains $(x_n)$ or $(y_1)$ and we iteratively take its minimal primes, their sums and the minimal primes of these sums. It can be shown that in this case we obtain prime ideals which can be written as the sum of an ideal generated by variables and an ideal $L$ with $L=\overline{L}$ as in Proposition \ref{PrIdCf}.\\
%
%
%
%
%
 
 Let $I$ be the sum of some minimal primes of $(f)$ that contains $(y_1)$ but not $(x_n)$. Then 
 $$I:=(y_1,f_{i_1,i_1+1},\ldots,f_{i_k ,i_k+1}).$$
 Note that we can always reduce the case $I=(y_1,f_{1,2},\ldots,f_{k-1, k})$. In fact, if $i_1 >1$ or if there exists an index $j$ such that $i_{j+1} \neq i_j +1$, then we can write $I$ as a sum of ideals on disjoint sets of variables, say $I= I_1+\ldots+I_t$. Hence the minimal primes of $I$ are sums of minimal primes of each $I_j$. Among them, the only minimal primes that we still have to study are those of the ideal $(y_1,f_{1,2},\ldots,f_{k-1, k})$ for some $k \in \{1, \ldots,n\}$. Analogously, if $I$ is the sum of some minimal primes of $(f)$ that contains $(x_n)$ but not $(y_1)$.\par 
 
 \begin{prop}\label{lemma1}
Let $I:=(y_1,f_{1,2},\ldots,f_{k-1, k})$ and let $P \in \Min (I)$. Then 
 \begin{enumerate}
 \item $P=(y_1, \ldots, y_k)$, or
 
 \item $P=\left(y_1,y_2, \ldots, y_{i-1},x_{i-1}\right)
+P_S $ for some $i \in \{2,\ldots,k\}$, where $P_S $ is a minimal prime of the binomial edge ideal of the path $$\{i,i+1\}, \ldots, \{k-1,k\}.$$ Hence, $P_S =\overline{P_S}$.
\end{enumerate}

Analogously, if $I:= (f_{k ,k+1}, \ldots,f_{n-1,n},x_n)$ and  $Q \in  Min (I)$ then 

\begin{enumerate}
\item $Q=\left( x_k, \ldots,x_n \right)$, or
 
\item $Q=\left( y_{l+1},x_{l+1},x_{l+2}, \ldots,x_n\right)+ Q_T$ where $Q_T $ is a minimal prime of the binomial edge ideal of the path $$\{k,k+1\} \ldots \{l-1,l\}.$$
Hence, $Q_T =\overline{Q_T}$.

\end{enumerate}
\end{prop}

\begin{proof}
We prove only the first part of the proposition. Simmetrically, one can prove the analogous result for $Q$. We start off  by noticing that 
$$I=(y_1, x_1 y_2,f_{2,3} \ldots, f_{k-1, k}).$$
Therefore, if $P$ is a minimal prime of $I$, we have that $P \supset (x_1 y_2)$. Since $P$ is prime, we have two possibilities, either $x_1 \in P$ or $y_2 \in P$.\par 
If $x_1 \in P$, then $P \supseteq (x_1,y_1)+ I \supseteq I$ and if we set $\tilde{I}=(x_1,y_1)+ I$, we have that $P \in \Min (\tilde{I})$. But $\tilde{I}= (x_1,y_1)+(f_{2,3},\ldots, f_{k-1,k})$ is a sum of two ideals generated by polynomials in disjoint sets of variables. Hence we conclude that $P= (x_1,y_1)+ P_S$ where $P_S$ is a minimal prime of the ideal $(f_{2,3},\ldots, f_{k-1,k})$.\par 
If $y_2 \in P$, then $P \supseteq (y_1,y_2)+ I \supseteq I$ and if we set, as before, $\tilde{I}=(x_1,y_1)+ I$, we get that $P \in \Min (\tilde{I})$. Thus we need to study the minimal primes of $\tilde{I}$. Again, we notice that
$$\tilde{I}=(y_1,y_2,x_2 y_3,f_{3,4},\ldots,f_{k-1,k}).$$
Thus, $P\supset \tilde{I} \supset (x_2 y_3)$ and, since $P$ is a prime ideal, we deduce that either $x_2 \in P$ or $y_3 \in P$. We can then iterate the previous argument and get the thesis.
\end{proof}

\begin{oss}\label{step1} More generally, if $P \in \Min \left( y_1,f_{1,2},\ldots,f_{n-1,n},x_n\right)$, then

\begin{enumerate}
\item $P=(y_1,x_1, \ldots,x_n)$ or $P=(y_1, \ldots,y_n,x_n)$, or  
 \item  $P=\left( y_1,y_2, \ldots, y_{i-1},x_{i-1} \right)
+P_S + \left( y_{l+1},x_{l+1},x_{l+2}, \ldots,x_n \right)$  
   
where $P_S$ is a minimal prime of the binomial edge ideal of the path
$$\{i,i+1\}, \ldots, \{l-1,l\}.$$

Hence, $P_S =\overline{P_S}$.\par
\end{enumerate}
\end{oss}

Putting together Lemma \ref{lemmaP+Q}, Proposition \ref{lemma1}, and Remark \ref{step1} we get the following result about the primary decomposition of sums of minimal primes of $(f)$.

\begin{lem}\label{step2} Let $P_1, \ldots,P_k$ be minimal primes of $ \left( f \right)$ and let $Q \in \Min (P_1+\ldots+P_k)$. Then
\begin{equation}\label{primiQ}
Q=\left( y_1,y_2, \ldots, y_{i-1},x_{i-1} \right)
+L + \left( y_{l+1},x_{l+1},x_{l+2}, \ldots,x_n \right)
\end{equation}
   
where $L =\overline{L} \subset \mathbb{K}[x_i,x_{i+1},\ldots,x_{l-1},x_l]$ is a minimal prime of the binomial edge ideal of a weakly closed graph and each of the three summands may possibly be the zero ideal.

\end{lem}

This argument can be generalized in order to see what happens if we take the sum of two prime ideals as in (\ref{primiQ}) in order to characterize all the minimal primes of Knutson ideals associated with $f$.\\

The next proposition generalizes Proposition \ref{lemma1}.

\begin{prop}\label{lemma2}
Consider two ideals of the form
\begin{align*}
P_1&=\left(y_1,\ldots,y_{i-1}\right)+\left(x_j\right)_{J_1 \subset \{1,\ldots, i-1\}}+P_{S_1}, \qquad \ P_{S_1}=\overline{P_{S_1}}\\ 
P_2&=\left(y_1,\ldots,y_{k-1}\right)+\left(x_j\right)_{J_2 \subset \{1,\ldots, k-1\}}+P_{S_2}, \qquad P_{S_2}=\overline{P_{S_2}} 
\end{align*}
where $P_{S_1} \subset \mathbb{K}[x_i,x_{i+1},\ldots,x_{n-1},x_n]$ and $P_{S_2}\subset \mathbb{K}[x_k,x_{k+1},\ldots,x_{n-1},x_n]$ are minimal primes of binomial edge ideals of weakly closed graphs. Let $P \in \Min \left(P_1+P_2\right)$. Then there exists an integer $l$ such that
$$P= \left(y_1,\ldots,y_{l}\right)+\left(x_u\right)_{U \subset \{1,\ldots, l\}}+L$$
with $L=\overline{L}$ in $ \mathbb{K}[x_{l+1},y_{l+1},\ldots,x_{n},y_{n}]$.\par 
Simmetrically, if we consider
\begin{align*}
P_1&=\left(x_i,\ldots,x_n\right)+\left(y_v\right)_{J_1 \subset \{i,\ldots, n\}}+P_{S_1}, \qquad \ P_{S_1}=\overline{P_{S_1}}\\ 
P_2&=\left(x_k,\ldots,x_n\right)+\left(y_v\right)_{J_2 \subset \{k,\ldots, n\}}+P_{S_2}, \qquad P_{S_2}=\overline{P_{S_2}} 
\end{align*}
where $P_{S_1} \subset \mathbb{K}[x_1,x_{2},\ldots,x_{i-1}]$ and $P_{S_2}\subset \mathbb{K}[x_1,x_{2},\ldots,x_{k-1}]$ are minimal primes of binomial edge ideals of weakly closed graphs. Let $P \in \Min \left(P_1+P_2\right)$. Then there esists an integer $l$ such that
$$P= \left(x_l,\ldots,x_n\right)+\left(y_v\right)_{V \subset \{l,\ldots, n\}}+L$$
with $L=\overline{L}$ in $ \mathbb{K}[x_{1},y_{1},\ldots,x_{l-1},y_{l-1}]$.\\
\end{prop}

\begin{proof}
We only prove the first part of the proposition. The second part follows by symmetry. 

    
By hypothesis $P_{S_1}= \overline{P_{S_1}}$ in $\mathbb{K}[x_i,x_{i+1},\ldots,x_{n-1},x_n]$ and $P_{S_2}=\overline{P_{S_2}}$ in $\mathbb{K}[x_k,x_{k+1},\ldots,x_{n-1},x_n]$. This means that we can write them as
\begin{align*}
P_{S_1}=\overline{P_{S_1}}=& \Bigl( \bigcup \limits_{s \in S_1}\{x_s,y_s\}\Bigr) +I_2(X_{[r_1,t_1])})+\ldots+I_2(X_{[r_{c_1},t_{c_1}]})\\
 P_{S_2}=\overline{P_{S_2}}=& \Bigl( \bigcup \limits_{s \in S_2}\{x_s,y_s\}\Bigr) +I_2(X_{[\tilde{r}_1,\tilde{t}_1])})+\ldots+I_2(X_{[\tilde{r}_{c_2},\tilde{t}_{c_2}]})
\end{align*}
 with $i\leq r_1 < t_1< r_{2} < t_{2} <\ldots< t_{c_1}\leq n$ and $ k\leq \tilde{r}_1 < \tilde{t}_1 <\tilde{r}_{2} < \tilde{t}_{2}<\ldots <\tilde{t}_{c_2}\leq n$. Without loss of generality we can assume that $i \leq k$. We want to study the minimal primes of the ideal $P_1+P_2$. First of all we notice that
  \begin{align*}
P_1+P_2 =&\left(y_l\right)_{l\leq k-1}+ \left(x_j\right)_{j \in J_1 \cup J_2}&& +\left( \bigcup \limits_{s \in S_1}\{x_s,y_s\}\right)+I_2(X_{[r_1,t_1])})+\ldots+I_2(X_{[r_{c_1},t_{c_1}]})\\
& &&+\left( \bigcup \limits_{s \in S_2}\{x_s,y_s\}\right) +I_2(X_{[\tilde{r}_1,\tilde{t}_1])})+\ldots+I_2(X_{[\tilde{r}_{c_2},\tilde{t}_{c_2}]})
\end{align*}
and, possibly dropping some summands, we can assume that $ t_1 \geq k$.\par 
But then, keeping in mind that $y_i, \ldots,y_{k-1} \in P_1+P_2$, we can write this sum as

\begin{align*}
P_1+P_2 &= \left(y_l\right)_{l\leq k-1}+ \left(x_j\right)_{j \in J_1 \cup J_2} +  \left( x_s\right)_{ \substack{{s \in S_1}\\{s \leq k-1}}}+\\
    &+\Bigl( \bigcup _{ \substack{{s \in S_1}\\{s \geq k}}} \{x_s,y_s\}\Bigr)+I_2(X_{[r'_1,t_1])})+I_2(X_{[r_{2},t_{2}]})+\ldots +I_2(X_{[r_{c_1},t_{c_1}]})+\\
    &+\left(x_u y_v\right)_{\substack {u \in  U_1:=\{r_1,\ldots,k-1\}\\ v \in V_1:=\{k,\ldots,t_1\}}}+P_{S_2}
\end{align*}
with $r'_1 \geq k$.\par

If $P \in \Min (P_1+P_2)$, then $P \supseteq Q_1$ for some $Q_1 \in \Min \left(x_u y_v\right)_{\substack {u \in  U_1 \\ v \in V_1}} $. Therefore $$P \in \Min \left(Q_1+\left(P_1+P_2\right)\right).$$ On the other hand, there are only two possibilities for $Q_1$
\begin{itemize}[itemsep=1pt,topsep=4pt] 
 \item[•]$Q_1=\left(x_u\right)_{u \in U_1}$
 \item[•]$Q_1=\left(y_v \right)_{v \in V_1}.$
\end{itemize}

\begin{itemize}[itemsep=1pt,topsep=4pt,leftmargin=0.1in] 
\item[]\textbf{1st case:} If $Q_1 =\left(x_u\right)_{u \in U_1} $, we have
 \begin{align*}
 Q_1 +\left(P_1+P_2\right)=\left(y_l\right)_{l\leq k-1}+ \left(x_j\right)_{ j \in U} +\left( P'_{S_1}+P_{S_2}\right) 
\end{align*}
where 
\begin{align*}
U&:= \left(J_1 \cup J_2 \cup S_1\cup U_1\right) \cap \{1, \ldots,k-1\}\\
P'_{S_1}&:= \Bigl( \bigcup _{ \substack{{s \in S_1}\\{s \geq k}}} \{x_s,y_s\}\Bigr) +I_2(X_{[r'_1,t_1])})+I_2(X_{[r_2,t_2])})+\ldots+I_2(X_{[r_{c_1},t_{c_1}]}) 
\end{align*}
   and $k \leq r'_1< t_1<r_2<\ldots<t_{c_1} \leq n$.   
 Thus, we have written $Q_1+\left(P_1+P_2\right)$ as the sum of two ideals of polynomials on disjoint sets of variables, namely 
 $\left(y_l\right)_{l\leq k-1}+ \left(x_j\right)_{ j \in U}$ and 
 $P'_{S_1}+P_{S_2}$. Since $P \in \Min \left(Q_1+\left(P_1+P_2\right)\right)$, we get that 
 $$P= \left(y_l\right)_{l\leq k-1}+ \left(x_u\right)_{U \subset \{1,\ldots, k-1\}}+L $$
 where $L  \in \Min (P'_{S_1}+P_{S_2})$. By Lemma \ref{lemmaP+Q}, we know that $L =\overline{L} \subseteq \mathbb{K}[x_k,\ldots,x_n]$ is a minimal prime of the binomial edge ideal of a weakly closed graph and we are done.\\

 \item[] \textbf{2nd case:} If  $Q_1 =\left(y_v \right)_{v \in V_1}$, then $I_2(X_{[r'_{1},t_{1}]})\in \left( y_l\right)_{l \leq t_1}$ and  we get
 \begin{align*}
 Q_1 +\left(P_1+P_2\right)&=\left( y_l\right)_{l \leq t_1}+\left(x_j\right)_{j \in J_1 \cup J_2} &\null&+\Bigl( \bigcup \limits_{s \in S_1}\{x_s,y_s\}\Bigr) +I_2(X_{[r_{2},t_{2}]})+\ldots +I_2(X_{[r_{c_1},t_{c_1}]})\\
 & &\null &+\Bigl( \bigcup \limits_{s \in S_2}\{x_s,y_s\}\Bigr) +I_2(X_{[\tilde{r}_1,\tilde{t}_1])})+\ldots+I_2(X_{[\tilde{r}_{c_2},\tilde{t}_{c_2}]}).
 \end{align*}
 
 where $r_2>t_1$ and again, possibly dropping some summands, we can assume that $ \tilde{t}_1 \geq t_1+1$.\par 
 As before, keeping in mind that $y_1, \ldots,y_{t_1} \in Q_1 +\left(P_1+P_2\right)$, we can write 
 \begin{align*}
Q_1+\left(P_1+P_2\right)&= \left(y_l\right)_{l\leq t_1}+ \left(x_j\right)_{j \in J_1 \cup J_2} +  \left( x_s\right)_{ \substack{{s \in S_1 \cup S_2}\\{s \leq t_1}}}+\\
    &+\Bigl( \bigcup _{ \substack{{s \in S_1}\\{s \geq t_1+1}}} \{x_s,y_s\}\Bigr)+I_2(X_{[r_{2},t_{2}]})+\ldots +I_2(X_{[r_{c_1},t_{c_1}]})+\\
    & +\Bigl( \bigcup _{ \substack{{s \in S_2}\\{s \geq t_1+1}}} \{x_s,y_s\}\Bigr)+I_2(X_{[{\tilde{r}_1}^{\prime},\tilde{t}_1])})+\ldots+I_2(X_{[\tilde{r}_{c_2},\tilde{t}_{c_2}]})\\
    &+\left(x_u y_v\right)_{\substack {u \in  \widetilde{U}_1:=\{\tilde{r}_1,\ldots,t_1\}\\ v \in \widetilde{V}_1:=\{t_1+1,\ldots,\tilde{t}_1\}}}  
\end{align*}
with ${\tilde{r}_1}^{\prime} \geq t_1+1$. Since $P \in \Min \left(Q_1+\left(P_1+P_2\right)\right)$ then $P$ must contain a minimal prime $Q_2$ of the ideal $\left(x_u y_v\right)_{\substack {u \in \widetilde{U}_1\\ v \in \widetilde{V}_1}}$. Hence $$P \in \Min \left(Q_2+ Q_1+\left(P_1+P_2\right) \right).$$ Again, there are only two possisbilities for $Q_2$, namely
  \begin{itemize}[itemsep=1pt,topsep=4pt] 
  \item[•]$Q_2=\left( x_u \right)_{u \in \widetilde{U}_1}$
  \item[•]$Q_2=\left( y_v \right)_{v \in \widetilde{V}_1}.$
\end{itemize}
and we can repeat the same argument as before.\par 
Thus, if $Q_2=\left( x_u \right)_{u \in \widetilde{U}_1}$ we have that
\begin{align*}
 Q_1 +\left(P_1+P_2\right)=\left(y_l\right)_{l\leq t_1}+ \left(x_j\right)_{ j \in U} +\left( P''_{S_1}+P'_{S_2}\right) 
\end{align*}
where 
\begin{align*}
U&:= \left(J_1 \cup J_2 \cup S_1 \cup S_2 \cup \widetilde{U}_1 \right) \cap \{1, \ldots,t_1\}\\
P''_{S_1}&:=\Bigl( \bigcup _{ \substack{{s \in S_1}\\{s \geq t_1+1}}} \{x_s,y_s\}\Bigr)+I_2(X_{[r_{2},t_{2}]})+\ldots +I_2(X_{[r_{c_1},t_{c_1}]})\\
P'_{S_2}&:= \Bigl( \bigcup _{ \substack{{s \in S_2}\\{s \geq t_1+1}}} \{x_s,y_s\}\Bigr)+I_2(X_{[{\tilde{r}_1}^{\prime},\tilde{t}_1])})+\ldots+I_2(X_{[\tilde{r}_{c_2},\tilde{t}_{c_2}]}).
\end{align*}
Since $r_2,{\tilde{r}_1}^{\prime}>t_1$, we have written $ Q_2+ Q_1+\left(P_1+P_2\right)$ as a sum of two ideals of polynomials on disjoint sets of variables, namely 
 $\left(y_l\right)_{l\leq t_1}+ \left(x_j\right)_{ j \in U}$ and 
 $P''_{S_1}+P'_{S_2}$. Since $P \in \Min \left(Q_2+Q_1+\left(P_1+P_2\right)\right)$, we get that 
 $$P= \left(y_l\right)_{l\leq t_1}+ \left(x_u\right)_{U \subset \{1,\ldots, t_1\}}+L $$
 where $L  \in \Min (P''_{S_1}+P'_{S_2})$. By Lemma \ref{lemmaP+Q}, we know that $L =\overline{L} \subseteq \mathbb{K}[x_{t_1+1},\ldots,x_n]$ is a minimal prime of a weakly closed graph  and we are done.\\

If instead, $Q_2=\left( y_v \right)_{v \in \widetilde{V}_1}$, we can iterate the above procedure.  At each step we add new variables $Q_p$ to the generators of $P_1+P_2$ but we still have $P \in \Min (Q_p+\ldots+Q_1+P_1+P_2)$. Since we have a finite number of variables, the algorithm must always terminate after a finite number of steps.
\end{itemize}
\end{proof}
From Proposition \ref{lemma2} and Lemma \ref{lemmaP+Q}, we obtain a generalization of Lemma \ref{step2}.

\begin{lem}\label{corP}
Let $P_1$ and $P_2$ be two prime ideals of the form
\begin{align*}
P_1&=\left(y_1,\ldots,y_{a-1}\right)+\left(x_u\right)_{U_1 \subset \{1,\ldots, a-1\}}+P_{S_1}+\left( x_{b+1},\ldots,x_n \right)+\left( y_v \right)_{V_1 \subset \{b+1,\ldots, n\}}\\ 
P_2&=\left(y_1,\ldots,y_{c-1}\right)+\left(x_u\right)_{U_2 \subset \{1,\ldots, c-1\}}+P_{S_2}+\left(x_{d+1},\ldots,x_n\right)+\left(y_v\right)_{V_2 \subset \{d+1,\ldots, n\}}
\end{align*}
where $P_{S_1}=\overline{P_{S_1}}\subset \mathbb{K}[x_a,x_{a+1}, \ldots, x_{b}]$ and $ P_{S_2}=\overline{P_{S_2}}\subset \mathbb{K}[x_c,x_{c+1}, \ldots, x_{c}]$ are minimal primes of binomial edge ideals of weakly closed graphs and each of the summands may possibly be the zero ideal. Let $P \in \Min \left(P_1+P_2\right)$. Then
$$P= \left(y_1,\ldots,y_{i-1}\right)+\left(x_u\right)_{U \subset \{1,\ldots, i-1\}}+L+\left(x_{j+1},\ldots,x_n\right)+\left(y_v\right)_{V \subset \{j+1,\ldots, n\}}$$
with $L=\overline{L}\subseteq \mathbb{K}[x_{i},y_{i},\ldots,x_{j},y_{j}]$ and each of the three summands may possibly be the zero ideal.
\end{lem}

Finally, putting together Lemma \ref{step2} and Lemma \ref{corP}, we manage to identify all the minimal primes of the ideals in $\mathcal{C}_f$ and prove Proposition \ref{PrIdCf}.


\begin{proof}
We recall that $\mathcal{C}_f$ is constructed from $(f)$ by taking its minimal primes, their sums, their intersections, and iterating. Since by Remark \ref{sumint}, finite sums and intersections commute, we only need to prove the result for sums of minimal prime ideals.\par  We know that the minimal primes of $(f)$ are
$$(x_n), (f_{1,2}),(f_{2,3}),\ldots,(f_{n-1,n}),(y_1).$$
These primes have the desired form.\par 
If we take the sum of some of these prime ideals, by Lemma \ref{step2}  the minimal primes of this sum are of the form
$$  P=\left( y_1,y_2, \ldots, y_{i-1},x_{i-1} \right)
+P_S + \left( y_{l+1},x_{l+1},x_{l+2}, \ldots,x_n \right)$$
   where $P_S= \overline{P_S}\subset \mathbb{K}[x_i,x_{i+1},\ldots,x_{l-1},x_l]$ is a minimal prime of the binomial edge ideal of a weakly closed graph.
Note that these minimal primes satisfy the hypotheses of Lemma \ref{corP}. Thus applying Lemma \ref{corP} and iterating this procedure we get the thesis.
\end{proof}

%

From Theorem \ref{cf-wc}, we know that Knutson binomial edge ideals are exactly those binomial edge ideals associated with weakly closed graphs. In view of this result one might hope to find a generalization of this theorem to higher dimensions, that is a characterization of all Knutson determinantal facet ideals.\par 
Determinantal facet ideals of simplicial complexes are the natural extension of binomial edge ideals of graphs. In \cite{BSV}, the authors introduce \lq \lq unit-interval”, \lq \lq under-closed”, and \lq \lq weakly-closed” simplicial complexes as natural $d$-dimensional generalizations of unit-interval, interval, and weakly-closed graphs and they investigate their  connections with Hamiltonian paths and determinantal facet ideals. \par 
Certainly, by \cite[Theorem 77]{BSV}, determinantal facet ideals of semi-closed simplicial complexes are Knutson ideals  and we know from \cite[Example 73]{BSV} that this result does not extend to weakly closed simplicial complexes. However, there could be other simplicial complexes whose determinantal facet ideals are Knutson ideals. \par 
This is a challenging question, somehow related to the problem of finding a primary decomposition of determinantal facet ideals. As we have seen, the proof of Theorem \ref{cf-wc} makes heavy use of primary decompositions of binomial edge ideals, which are well-known. Instead, the primary decomposition of a determinantal facet ideal is still unknown in general, even if there have been some steps in this direction (\cite{HS}, \cite{MR}). This makes it hard to generalize the previous proofs to higher dimensions.

 \end{document}